\documentclass[12pt]{amsart}
\usepackage[english]{babel}
\usepackage[utf8]{inputenc}
\usepackage{lipsum}
\usepackage{mathrsfs}
\usepackage{stix}
\usepackage{fullpage}
\usepackage{mathtools}
\usepackage{amsmath}
\usepackage{tikz-cd}
\usepackage[colorlinks]{hyperref}
\usepackage{fullpage}
\usepackage{mathtools}
\usepackage{amsmath}
\usepackage{tikz-cd}
\numberwithin{equation}{section}
\theoremstyle{plain}
\newtheorem{thm}{Theorem}[section]
\newtheorem{prop}[thm]{Proposition}

\newtheorem{rem}[thm]{Remark}

\def\X{\mathbb{X}}
\def\R{\mathbb{R}^{n}}
\def\G{\mathbb{G}}
\def\S{\mathfrak{S}}
\begin{document}

%%%% Article title to be placed here
\title{Hardy inequalities on metric measure spaces, III: The case $q\leq p<0$ and applications}

\author[A. Kassymov]{Aidyn Kassymov}
\address{
  Aidyn Kassymov:
  \endgraf
   \endgraf
  Department of Mathematics: Analysis, Logic and Discrete Mathematics
  \endgraf
  Ghent University, Belgium
  \endgraf
  and
  \endgraf
  Institute of Mathematics and Mathematical Modeling
  \endgraf
  125 Pushkin str.
  \endgraf
  050010 Almaty
  \endgraf
  Kazakhstan
  \endgraf
  and
  \endgraf
  Al-Farabi Kazakh National University
  \endgraf
   71 Al-Farabi avenue
   \endgraf
   050040 Almaty
   \endgraf
   Kazakhstan
  \endgraf
	{\it E-mail address} {\rm aidyn.kassymov@ugent.be} and {\rm kassymov@math.kz}}

  \author[M. Ruzhansky]{Michael Ruzhansky}
\address{
	Michael Ruzhansky:
	 \endgraf
  Department of Mathematics: Analysis, Logic and Discrete Mathematics
  \endgraf
  Ghent University, Belgium
  \endgraf
  and
  \endgraf
  School of Mathematical Sciences
    \endgraf
    Queen Mary University of London
  \endgraf
  United Kingdom
	\endgraf
  {\it E-mail address} {\rm michael.ruzhansky@ugent.be}}

\author[D. Suragan]{Durvudkhan Suragan}
\address{
	Durvudkhan Suragan:
	\endgraf
	Department of Mathematics
	\endgraf
	School of Science and Technology, Nazarbayev University
    \endgraf
	53 Kabanbay Batyr Ave, Nur-Sultan 010000
	\endgraf
	Kazakhstan
	\endgraf
	{\it E-mail address} {\rm durvudkhan.suragan@nu.edu.kz}}

\thanks{
The first and second authors were supported in parts by the FWO Odysseus 1 grant no. G.0H94.18N: Analysis and Partial Differential Equations, by the Methusalem programme of the Ghent University Special Research Fund (BOF) (grant no. 01M01021) and by the EPSRC grant  EP/R003025/2. Also, this research has been funded by the Science Committee of the Ministry of Science and Higher Education of the Republic of Kazakhstan (Grant No.AP19676031) and  partially supported by the collaborative research program "Qualitative analysis for nonlocal and fractional models" from Nazarbayev University. }


%%%% Abstract text to be placed here %%%%%%%%%%%%
\begin{abstract}
In this paper, we obtain a reverse version of the integral Hardy inequality on metric measure space with two negative exponents.  Also, as for applications we show the reverse Hardy-Littlewood-Sobolev and the Stein-Weiss  inequalities  with two negative exponents on homogeneous Lie groups and with arbitrary quasi-norm, the result which appears to be new already in the Euclidean space. This work further complements the ranges of $p$ and $q$ (namely, $q\leq p<0$) considered in \cite{RV} and \cite{RV21}, where one treated the cases $1<p\leq q<\infty$ and $p>q$, respectively.
\end{abstract}\maketitle
%%%%%%%%%%%%%%%%%%%%%%%%%%%

%%%%%%%%%% Insert the texts which can accomdate on firstpage in the tag "fmtext" %%%%%

\section{Introduction}
In the famous work \cite{Har20}, G.H. Hardy showed the following (direct) integral inequality:
\begin{align}
\int_{a}^{\infty}\frac{1}{x^{p}}\left(\int_{a}^{\infty}f(t)dt\right)^{p}dx\leq\left(\frac{p}{p-1}\right)^{p}\int_{a}^{\infty}f^{p}(x)dx,
\end{align}
where $f\geq0$, $p>1$, and $a>0$. The subject of the Hardy inequalities has been extensively investigated and
we refer to the book \cite{KMP}. 

We refer to  direct inequalities \cite{Dav99,DHK,EE04, GKPW04, KMP, KP03, KPS17, OK90} and to the reverse inequalities \cite{BH,GKK,KKK08,KK, Pro}.

 The main goal of this  paper is to extend the reverse Hardy inequalities to general metric measure space with two negative exponents. More specifically, we consider metric spaces $\mathbb X$ with a Borel measure $dx$ allowing for the following {\em polar decomposition} at $a\in{\mathbb X}$: we assume that there is a locally integrable function $\lambda \in L^1_{loc}$  such that for all $f\in L^1(\mathbb X)$ we have
   \begin{align}\label{EQ:polarintro}
   \int_{\mathbb X}f(x)dx= \int_0^{\infty}\int_{\Sigma_{r}} f(r,\omega) \lambda(r,\omega) d\omega_{r} dr,
   \end{align}
    for some set $\Sigma_{r}=\{x\in\mathbb{X}:d(x,a)=r\}\subset \mathbb X$ with a measure on it denoted by $d\omega$, and $(r,\omega)\rightarrow a $ as $r\rightarrow0$.

 The condition (1.2) is rather general (see \cite{RV}) since we allow the function $\lambda$ to depend on the whole variable $x=(r,\omega)$. Since $\mathbb X$ does not necessarily have a differentiable structure, the function $\lambda(r,\omega)$ can not be in general obtained as the Jacobian of the polar change of coordinates. However, if such a differentiable structure exists on $\mathbb X$, the condition (1.2) can be obtained as the standard polar decomposition formula.
In particular, let us give several examples of $\mathbb X$ for which the condition (1.2) is satisfied with different expressions for $ \lambda (r,\omega)$:

\begin{itemize}
\item[(I)] Euclidean space $\mathbb{R}^{n}$: $ \lambda (r,\omega)= {r}^{n-1}.$
\item[(II)] Homogeneous groups: $ \lambda (r,\omega)= {r}^{Q-1}$, where $Q$ is the homogeneous dimension of the group. Such groups have been consistently developed by Folland and Stein \cite{FS1}, see also an up-to-date exposition in \cite{FR} and \cite{RY18b}.
\item[(III)] Hyperbolic spaces $\mathbb H^n$:  $\lambda(r,\omega)=(\sinh {r})^{n-1}$.
\item[(IV)] Cartan-Hadamard manifolds: Let $K_M$ be the sectional curvature on $(M, g).$ A Riemannian manifold $(M, g)$ is called {\em a Cartan-Hadamard manifold} if it is complete, simply connected and has non-positive sectional curvature, i.e., the sectional curvature $K_M\le 0$ along each plane section at each point of $M$. Let us fix a point $a\in M$ and denote by
$\rho(x)=d(x,a)$ the geodesic distance from $x$ to $a$ on $M$. The exponential map ${\rm exp}_a :T_a M \to  M$ is a diffeomorphism, see e.g. Helgason \cite{DV3}.  Let $J(\rho,\omega)$ be the density function on $M$, see e.g. \cite{DV1}. Then we have the following polar decomposition:
$$
\int_M f(x) dx=\int_0^{\infty}\int_{\mathbb S^{n-1}}f({\rm exp}_{a}(\rho \omega))J(\rho,\omega) \rho^{n-1}d\rho d\omega,
$$
so that we have (1.2) with $\lambda(\rho,\omega)= J(\rho,\omega) \rho^{n-1}.$
\end{itemize}
 In \cite{RV} and \cite{RV21}, the (direct) integral Hardy inequality on metric measure spaces was established with applications to homogeneous Lie groups, hyperbolic spaces, Cartan-Hadamard manifolds with negative curvature and on general Lie groups with Riemannian distance for  $1<p\leq q<\infty$ and $p>q$, respectively. Also, in \cite{KRSin}, the authors showed the integral Hardy inequality for $p\in(0,1)$ and $q<0$ on metric measure space.   In this paper, we continue the investigation of the integral Hardy inequality on a metric measure space, i.e., we show the reverse integral Hardy inequality with negative exponents. 
%Main motivation of this paper is a combine results in \cite{RV} and reversed inequalities.

In \cite{HL28}, Hardy and Littlewood considered  the one dimensional fractional integral operator on $(0,\infty)$ given by
\begin{equation}\label{1Doper}
T_{\lambda}u(x)=\int_{0}^{\infty}\frac{u(y)}{|x-y|^{\lambda}}dy,\,\,\,\,0<\lambda<1,
\end{equation}
where they also showed the following $L^{q}-L^{p}$ boundedness of this operator $T_{\lambda}$:
\begin{thm}\label{1DHLS28}
Let $1<p<q<\infty$ and $u\in L^{p}(0,\infty)$  with $\frac{1}{q}=\frac{1}{p}+\lambda-1$. Then
\begin{equation}
\|T_{\lambda}u\|_{L^{q}(0,\infty)}\leq C \|u\|_{L^{p}(0,\infty)},
\end{equation}
where $C$ is a positive constant independent of $u$.
\end{thm}
The multi-dimensional analogue of (1.3) can be represented by the formula:
\begin{equation}\label{NDoper}
I_{\lambda}u(x)=\int_{\mathbb{R}^{N}}\frac{u(y)}{|x-y|^{\lambda}}dy,\,\,\,\,0<\lambda<N.
\end{equation}
In  \cite{Sob38},  Sobolev generalised  Theorem 1.1 for multi-dimensional case in the following form:
\begin{thm}\label{THM:HLS}
Let $1<p<q<\infty$, $u\in L^{p}(\mathbb{R}^{N})$  with $\frac{1}{q}=\frac{1}{p}+\frac{\lambda}{N}-1$. Then
\begin{equation}
\|I_{\lambda}u\|_{L^{q}(\mathbb{R}^{N})}\leq C \|u\|_{L^{p}(\mathbb{R}^{N})},
\end{equation}
where $C$ is a positive constant independent of $u$.
\end{thm}
In \cite{StWe58}, Stein and Weiss obtained the following radially weighted Hardy-Littlewood-Sobolev inequality, which is known as the Stein-Weiss inequality. 
\begin{thm}\label{Classiacal_Stein-Weiss_inequality}
Let $0<\lambda<N$, $1<p<\infty$, $\alpha<\frac{N(p-1)}{p}$, $\beta<\frac{N}{q}$, $\alpha+\beta\geq0$ and $\frac{1}{q}=\frac{1}{p}+\frac{\lambda+\alpha+\beta}{N}-1$. If $1<p\leq q<\infty$, then
\begin{equation}\label{stein-weissclas}
\||x|^{-\beta}I_{\lambda}u\|_{L^{q}(\mathbb{R}^{N})}\leq C \||x|^{\alpha}u\|_{L^{p}(\mathbb{R}^{N})},
\end{equation}
where $C$ is a positive constant independent of $u$.
\end{thm}

 To the best of our knowledge,   the Hardy-Littlewood-Sobolev inequality on the Heisenberg group was proved by Folland and Stein in \cite{FS74} and the best constants of the Hardy-Littlewood-Sobolev inequality, in the Euclidean space and Heisenberg group were obtained in \cite{Lie83} and \cite{FL12}, respectively.  Also, in \cite{HLZ}, \cite{RY18b} and \cite{KRS}, the authors studied the Hardy-Littlewood-Sobolev and the Stein-Weiss inequalities on Heisenberg and homogeneous Lie groups. Note that  systematic studies of different functional inequalities on general homogeneous (Lie)
groups were initiated by the papers \cite{ORS,RS17, RSY18,  RY18a}.

The reverse Stein-Weiss inequality in Euclidean setting has the following form:
\begin{thm}[\cite{CLT1}, Theorem 1]
For $n\geq1, p\in(0,1), q<0, \lambda>0, 0\leq\alpha<-\frac{n}{q}$, and $0\leq \beta<-\frac{n}{p'}$ satisfying $\frac{1}{p}+\frac{1}{q'}-\frac{\alpha+\beta+\lambda}{n}=2$, there is a constant $C=C(n,\alpha,\beta,\lambda,p,q)>0$ such that for any non-negative functions $f\in L^{q'}(\mathbb{R}^{n})$ and $0<\int_{\mathbb{R}^{n}}g^{p}(y)dy<\infty$, we have
\begin{equation}\label{cltst1}
\int_{\mathbb{R}^{n}}\int_{\mathbb{R}^{n}}|x|^{\alpha}|x-y|^{\lambda}f(x)g(y)|y|^{\beta}dydx\geq C\left(\int_{\mathbb{R}^{n}}f^{q'}(x)dx\right)^{\frac{1}{q'}}\left(\int_{\mathbb{R}^{n}}g^{p}(y)dy\right)^{\frac{1}{p}},
\end{equation}
where $\frac{1}{q}+\frac{1}{q'}=1$ and $\frac{1}{p}+\frac{1}{p'}=1$.
\end{thm}
Note,  we obtain the reverse Hardy-Littlewood-Sobolev inequality if $\alpha=\beta=0$. Improved Stein-Weiss inequality was obtained in \cite{CLT} on the Euclidean upper half-space and in \cite{KRS2} on homogeneous Lie groups.  For more results about the reverse Hardy–Littlewood–Sobolev  inequality in Euclidean space, we refer the reader to  \cite{B15} \cite{CDDFF}, \cite{JZ}, \cite{NN}   and the references therein.  Note that the reverse  Hardy-Littlewood-Sobolev and Stein-Weiss inequalities were shown in \cite{KRS2} for the case $p\in(0,1)$ and $q<0$. In this paper, we show the reverse Hardy-Littlewood-Sobolev and  Stein-Weiss  inequalities with two negative exponents i.e., $q<p<0$, which is also new in the Euclidean space.
\section{Main result}
 
Firstly, let us denote by $B(a, r)$ a ball in $\X$ with centre $a$ and radius $r$, i.e.,
$$B(a, r) := \{x \in \X : d(x,a) < r\},$$
where $d$ is the metric on $\X$. Once and for all let us fix some point $a \in \X$, and denote
\begin{equation}
|x|_{a}:= d(a, x).
\end{equation}
Let us recall briefly the reverse H\"{o}lder inequality.
\begin{thm}[\cite{AF}, Theorem 2.12, p. 27]\label{Hol}
Let $p<0$, so that $p'=\frac{p}{p-1}>0$. If non-negative functions satisfy $0<\int_{\X}f^{p}(x)dx<+\infty$ and $0<\int_{\X}g^{p'}(x)dx<+\infty,$ we have
\begin{equation}\label{Holin}
\int_{\X}f(x)g(x)dx\geq\left(\int_{\X}f^{p}(x)dx\right)^{\frac{1}{p}}\left(\int_{\X}g^{p'}(x)dx\right)^{\frac{1}{p'}}.
\end{equation}
\end{thm}
As the main results of this section, we show the reverse integral Hardy inequality as well as its conjugate.
\begin{thm}\label{integralhar1}
Assume that $p,q<0$ are such that $q\leq p<0$. Let $\mathbb X$ be a metric measure space with a polar decomposition at $a\in \X$. Suppose that $u,v\geq0$ are locally integrable functions on $\mathbb X$. Then the inequality
\begin{equation}\label{eqinteg}
\left[\int_{\mathbb X}\left(\int_{B(a,|x|_{a})}f(y)dy\right)^{q}u(x)dx\right]^{\frac{1}{q}}\geq C_{1}(p,q)\left(\int_{\mathbb X}f^{p}(x)v(x)dx\right)^{\frac{1}{p}}
\end{equation}
holds for all non-negative real-valued measurable functions $f$, if and only if
\begin{equation}\label{D1}
0< D_{1}=\inf_{x\neq a}\mathcal{D}_{1}(|x|_{a})=\inf_{x\neq a}\left[\left(\int_{B(a,|x|_{a})}u(y)dy\right)^{\frac{1}{q}}\left(\int_{B(a,|x|_{a})}v^{1-p'}(y)dy\right)^{\frac{1}{p'}}\right],
\end{equation}
and $\mathcal{D}_{1}(|x|_{a})$ is  non-decreasing.
Moreover, the largest constant $C_{1}(p,q)$ in (2.3) satisfies 
\begin{equation}\label{ocenc1}
    D_{1}\geq C_{1}(p,q)\geq |p|^{\frac{1}{q}}(p')^{\frac{1}{p'}}D_{1},
\end{equation}
where $\frac{1}{p}+\frac{1}{p'}=1.$
\end{thm}
\begin{rem}
In (2.5), by simple calculation, we have that the for the case $q\leq p<0$
\begin{equation}
    |p|^{\frac{1}{q}}(p')^{\frac{1}{p'}}\leq 1.
\end{equation}
\end{rem}
\begin{proof}[Proof of Theorem 2.2]
Let us divide a proof of this theorem to 2 steps.

\textbf{Step 1.}  Firstly, let us denote
\begin{align}
&F(s):=\int_{\sum_{s}}\lambda(s,\sigma)f^{p}(s,\sigma)v(s,\sigma)d\sigma,\label{fn}\\&
V(s):=\int_{\sum_{s}}\lambda(s,\sigma)v^{1-p'}(s,\sigma)d\sigma,\label{vn}\\&
h(t):=\left(\int_{0}^{t}\int_{\sum_{s}}\lambda(s,\sigma)v^{1-p'}(s,\sigma)d\sigma ds\right)^{\frac{1}{pp'}},\label{h}\\&
H_{1}(t):=\int_{0}^{t}\int_{\sum_{s}}\lambda(s,\sigma)v^{-\frac{p'}{p}}(s,\sigma)h^{-p'}(s)d\sigma ds,\label{h1}\\&
U_{1}(s):=\int_{\sum_{s}}\lambda(s,\sigma)u(s,\sigma)d\sigma\label{u1}.
\end{align}

 By using the reverse H\"{o}lder inequality (2.2) with the polar decomposition, we compute
\begin{equation}\label{1mom}
\begin{split}
\int_{B(a,|x|_{a})}f(y)dy&=\int_{B(a,|x|_{a})}[f(y)v^{\frac{1}{p}}(y)h(y)][v^{\frac{1}{p}}(y)h(y)]^{-1}dy\\&
\geq \left(\int_{B(a,|x|_{a})}(f(y)v^{\frac{1}{p}}(y)h(y))^{p}dy\right)^{\frac{1}{p}} \left(\int_{B(a,|x|_{a})}(v^{\frac{1}{p}}(y)h(y))^{-p'}dy\right)^{\frac{1}{p'}}\\&
=\left(\int_{0}^{|x|_{a}}\int_{\sum_{s}}h^{p}(s)\lambda(s,\sigma)f^{p}(s,\sigma)v(s,\sigma)d\sigma ds\right)^{\frac{1}{p}}\\&
\times
\left(\int_{0}^{|x|_{a}}\int_{\sum_{s}}v^{-\frac{p'}{p}}(s,\sigma)h^{-p'}(s)\lambda(s,\sigma)d\sigma ds\right)^{\frac{1}{p'}}\\&
=\left(\int_{0}^{|x|_{a}}h^{p}(s)F(s) ds\right)^{\frac{1}{p}}H_{1}^{\frac{1}{p'}}(|x|_{a}).
\end{split}
\end{equation}

Let us calculate  $H_{1}(t)$:
\begin{equation}\label{2mom}
    \begin{split}
    &H_{1}(t)= \int_{0}^{t}\int_{\sum_{s}}\lambda(s,\sigma)v^{-\frac{p'}{p}}(s,\sigma)h^{-p'}(s)d\sigma ds\\&
    \stackrel{(2.8)}=\int_{0}^{t}h^{-p'}(s)V(s)ds\\&
    \stackrel{(2.9)}=\int_{0}^{t}\left(\int_{0}^{s}\int_{\sum_{z}}\lambda(z,\omega)v^{1-p'}(z,\omega)dzd\omega\right)^{-\frac{1}{p}}V(s)ds\\&
  \stackrel{(2.8)}=\int_{0}^{t}\left(\int_{0}^{s}V(z)dz\right)^{-\frac{1}{p}}V(s)ds\\&
  =\int_{0}^{t}\left(\int_{0}^{s}V(z)dz\right)^{-\frac{1}{p}}d_{s}\left(\int_{0}^{s}V(z)dz\right)\\&
 =p'\left(\int_{0}^{s}V(z)dz\right)^{\frac{1}{p'}}\big{|}_{0}^{t}\\&
    \stackrel{\frac{1}{p'}>0}=p'\left(\int_{0}^{t}V(z)dz\right)^{\frac{1}{p'}}\\&
    =p'h^{p}(t).
    \end{split}
\end{equation}
By combining  (2.13) and (2.12), we get
\begin{equation}\label{3mom}
\begin{split}
    \int_{B(a,|x|_{a})}f(y)dy&\geq\left(\int_{0}^{|x|_{a}}h^{p}(s)F(s) ds\right)^{\frac{1}{p}}H_{1}^{\frac{1}{p'}}(|x|_{a})\\&
   \stackrel{(2.13)}=(p')^{\frac{1}{p'}}\left(\int_{0}^{|x|_{a}}h^{p}(s)F(s) ds\right)^{\frac{1}{p}}h^{\frac{p}{p'}}(|x|_{a}).   
\end{split}
\end{equation}
Multiplying by $u$, integrating over $\X$ with $q<0$ and by using (direct) Minkowski's inequality with $\frac{q}{p}\geq1$ (see \cite{AF}, Theorem 2.9, p.26), we compute
 \begin{equation}\label{2.12}
     \begin{split}
         &\int_{\X} \left(\int_{B(a,|x|_{a})}f(y)dy\right)^{q}u(x)dx\\&
         =\int_{0}^{\infty}\int_{\sum_{r}}u(z,\omega)\lambda(z,\omega)\left(\int_{0}^{|x|_{a}}\int_{\sum_{s}}\lambda(s,\sigma)f(s,\sigma)dsd\sigma\right)^{q}dzd\omega\\&
         \stackrel{(2.11)}=\int_{0}^{\infty}U_{1}(z)\left(\int_{0}^{z}\int_{\sum_{s}}\lambda(s,\sigma)f(s,\sigma)dsd\sigma\right)^{q}dz\\&
         \stackrel{q<0,\,(2.14)}\leq(p')^{\frac{q}{p'}}\int_{0}^{\infty}U_{1}(z)\left(\int_{0}^{z}h^{p}(s)F(s) ds\right)^{\frac{q}{p}}h^{\frac{qp}{p'}}(z)dz\\&
         =(p')^{\frac{q}{p'}}\int_{0}^{\infty}U_{1}(z)\left(\int_{0}^{\infty}\chi_{[0,z]}h^{p}(s)F(s) ds\right)^{\frac{q}{p}}h^{\frac{qp}{p'}}(z)dz\\&
         \leq (p')^{\frac{q}{p'}}\left[\int_{0}^{\infty}h^{p}(s)F(s) \left(\int_{s}^{\infty}U_{1}(z)h^{\frac{qp}{p'}}(z)dz\right)^{\frac{p}{q}}ds\right]^{\frac{q}{p}},
     \end{split}
 \end{equation}
 where $\chi_{[0,r]}$ is the cut-off function.
 At the same time, one can also estimate
 \begin{equation}\label{hpq}
     \begin{split}
         h^{\frac{pq}{p'}}(t)&=\left[\left(\int_{0}^{t}\int_{\sum_{s}}\lambda(s,\sigma)v^{1-p'}(s,\sigma)dsd\sigma\right)^{\frac{q}{p'}}\right]^{\frac{1}{p'}}\\&
         \stackrel{(2.8)}=\left[\left(\int_{0}^{t}V(s)ds\right)^{\frac{q}{p'}}\right]^{\frac{1}{p'}}\\&
         =\left[\left(\int_{0}^{t}V(s)ds\right)^{\frac{q}{p'}}\left(\int_{0}^{t}U_{1}(s)ds\right)\left(\int_{0}^{t}U_{1}(s)ds\right)^{-1}\right]^{\frac{1}{p'}}\\&
         =\mathcal{D}^{\frac{q}{p'}}_{1}(|t|_{a})\left(\int_{0}^{t}U_{1}(s)ds\right)^{-\frac{1}{p'}},
     \end{split}
 \end{equation}
 where $\mathcal{D}_{1}(|t|_{a}):=\left(\int_{0}^{t}V(s)ds\right)^{\frac{1}{p'}}\left(\int_{0}^{t}U_{1}(s)ds\right)^{\frac{1}{q}}$.
 By using this fact and since  $\mathcal{D}_{1}(|x|_{a})$ is non-decreasing, we get
 \begin{equation*}
     \begin{split}
         &\int_{\X} \left(\int_{B(a,|x|_{a})}f(y)dy\right)^{q}u(x)dx\\&
         \stackrel{(2.15)}\leq (p')^{\frac{q}{p'}}\left[\int_{0}^{\infty}h^{p}(s)F(s) \left(\int_{s}^{\infty}U_{1}(r)h^{\frac{qp}{p'}}(r)dr\right)^{\frac{p}{q}}ds\right]^{\frac{q}{p}}\
            \end{split}
 \end{equation*}
  \begin{equation}
     \begin{split}
 &\stackrel{\frac{p}{q}>0,\,(2.16)}\leq (p')^{\frac{q}{p'}}\left[\int_{0}^{\infty}h^{p}(s)F(s)\mathcal{D}^{\frac{p}{p'}}_{1}(s) \left(\int_{s}^{\infty}U_{1}(r)\left(\int_{0}^{r}U_{1}(z)dz\right)^{-\frac{1}{p'}}dr\right)^{\frac{p}{q}}ds\right]^{\frac{q}{p}}\\&
 =(p')^{\frac{q}{p'}}\left[\int_{0}^{\infty}h^{p}(s)F(s)\mathcal{D}^{\frac{p}{p'}}_{1}(s) \left(\int_{s}^{\infty}d_{r}\left[p\left(\int_{0}^{r}U_{1}(z)dz\right)^{\frac{1}{p}}\right]\right)^{\frac{p}{q}}ds\right]^{\frac{q}{p}}\\&
         =(p')^{\frac{q}{p'}}\left[\int_{0}^{\infty}h^{p}(s)F(s)\mathcal{D}^{\frac{p}{p'}}_{1}(s) \left(p\left(\int_{0}^{\infty}U_{1}(z)dz\right)^{\frac{1}{p}}-p\left(\int_{0}^{s}U_{1}(z)dz\right)^{\frac{1}{p}}\right)^{\frac{p}{q}}ds\right]^{\frac{q}{p}}\\&
       \stackrel{p<0}\leq (-p)(p')^{\frac{q}{p'}}\left[\int_{0}^{\infty}h^{p}(s)F(s)\mathcal{D}^{\frac{p}{p'}}_{1}(s) \left(\int_{0}^{s}U_{1}(z)dz\right)^{\frac{1}{q}}ds\right]^{\frac{q}{p}}\\&
        \stackrel{(2.9)}=(-p)(p')^{\frac{q}{p'}}\left[\int_{0}^{\infty}F(s)\mathcal{D}^{1+\frac{p}{p'}}_{1}(s) ds\right]^{\frac{q}{p}}\\&
  =(-p)(p')^{\frac{q}{p'}}\left[\int_{0}^{\infty}F(s)\mathcal{D}^{p}_{1}(s) ds\right]^{\frac{q}{p}}\\&
         \stackrel{p<0}\leq (-p)(p')^{\frac{q}{p'}}D^{q}_{1}\left[\int_{0}^{\infty}F(s) ds\right]^{\frac{q}{p}}\\&
         \stackrel{(2.7)}=(-p)(p')^{\frac{q}{p'}}D^{q}_{1}\left(\int_{\X}f^{p}(x)v(x)dx\right)^{\frac{q}{p}}\\&
         =|p|(p')^{\frac{q}{p'}}D^{q}_{1}\left(\int_{\X}f^{p}(x)v(x)dx\right)^{\frac{q}{p}}.
     \end{split}
 \end{equation}
 Finally, we obtain
 \begin{equation}
 \left(\int_{\X} \left(\int_{B(a,|x|_{a})}f(y)dy\right)^{q}u(x)dx\right)^{\frac{1}{q}}\geq |p|^{\frac{1}{q}}(p')^{\frac{1}{p'}}D_{1}\left(\int_{\X}f^{p}(x)v(x)dx\right)^{\frac{1}{p}}.
  \end{equation}
  Hence, it follows that (2.3) holds with $C_{1}(p,q)\geq|p|^{\frac{1}{q}}(p')^{\frac{1}{p'}}D_{1}$, proving one of the relations in (2.5).
  
\textbf{Step 2.}
Now it remains to show that (2.3) yields (2.4). Let us fix $t>0$ and denote  the following function:
\begin{equation}
f(x):=
\begin{cases}
v^{1-p'}(x),\,\,\text{if}\,\,\,|x|_{a}\leq t,\\
\alpha f_{1}(x), \,\,\text{if}\,\,\,|x|_{a}> t,
\end{cases}
\end{equation}
 where
$f_{1}$ is any function satisfying $\int_{B(a,|x|_{a})}f_{1}(y)dy<\infty$ and $\int_{|x|_{a}\geq t}v(x)f^{p}_{1}(x)dx<\infty,$ and $\alpha>0$. Then we compute
\begin{equation*}
\begin{split}
C_{1}(p,q)&\leq \left[\int_{\mathbb X}\left(\int_{|y|_{a}\leq |x|_{a}}f(y)dy\right)^{q}u(x)dx\right]^{\frac{1}{q}} \left[\int_{\mathbb X}f^{p}(y)v(y)dy\right]^{-\frac{1}{p}}
\end{split}
\end{equation*}
\begin{equation*}
\begin{split}
&=\left[\int_{\mathbb X}\left(\int_{|y|_{a}\leq |x|_{a}}f(y)dy\right)^{q}u(x)dx\right]^{\frac{1}{q}}\left[\int_{|y|_{a}\leq t}v^{1-p'}(y)dy+\alpha^{p}\int_{|y|_{a}> t}v(y)f_{1}^{p}(y)dy\right]^{-\frac{1}{p}}\\&
\stackrel{q<0}\leq \left[\int_{|x|_{a}\leq t}\left(\int_{|y|_{a}\leq |x|_{a}}f(y)dy\right)^{q}u(x)dx\right]^{\frac{1}{q}}\left[\int_{|y|_{a}\leq t}v^{1-p'}(y)dy+\alpha^{p}\int_{|y|_{a}> t}v(y)f_{1}^{p}(y)dy\right]^{-\frac{1}{p}}\\&
=\left[\int_{|x|_{a}\leq t}\left(\int_{|y|_{a}\leq |x|_{a}}v^{1-p'}(y)dy\right)^{q}u(x)dx\right]^{\frac{1}{q}}\left[\int_{|y|_{a}\leq t}v^{1-p'}(y)dy+\alpha^{p}\int_{|y|_{a}> t}v(y)f_{1}^{p}(y)dy\right]^{-\frac{1}{p}}\\&
\stackrel{q<0}\leq\left[\int_{|x|_{a}\leq t}\left(\int_{|y|_{a}\leq t}v^{1-p'}(y)dy\right)^{q}u(x)dx\right]^{\frac{1}{q}}\left[\int_{|y|_{a}\leq t}v^{1-p'}(y)dy+\alpha^{p}\int_{|y|_{a}> t}v(y)f_{1}^{p}(y)dy\right]^{-\frac{1}{p}}\\&
=\left[\int_{|x|_{a}\leq t}u(x)dx\right]^{\frac{1}{q}}\left[\int_{|y|_{a}\leq t}v^{1-p'}(y)dy\right]\left[\int_{|y|_{a}\leq t}v^{1-p'}(y)dy+\alpha^{p}\int_{|y|_{a}> t}v(y)f_{1}^{p}(y)dy\right]^{-\frac{1}{p}}.
\end{split}
\end{equation*}
Summarising above facts with $q\leq p<0$ and taking limit as $\alpha\rightarrow 0$, we obtain
\begin{equation}
    C_{1}(p,q)\leq\left[\int_{|y|_{a}\leq t}v^{1-p'}(y)dy\right]^{\frac{1}{p'}}\left[\int_{|x|_{a}\leq t}u(x)dx\right]^{\frac{1}{q}}.
\end{equation}
Finally, we get $C_{1}(p,q)\leq D_{1}.$
\end{proof}
Now let us prove the conjugate integral Hardy inequality.
\begin{thm}\label{integralhar2}
Assume that $p,q<0$ such that $q\leq p<0$. Let $\mathbb X$ be a metric measure space with a polar decomposition at $a\in \X$. Suppose that $u,v\geq0$ are locally integrable functions on $\mathbb X$. Then the inequality
\begin{equation}\label{eqinteg2}
\left[\int_{\mathbb X}\left(\int_{\X\setminus B(a,|x|_{a})}f(y)dy\right)^{q}u(x)dx\right]^{\frac{1}{q}}\geq C_{2}(p,q)\left(\int_{\mathbb X}f^{p}(x)v(x)dx\right)^{\frac{1}{p}}
\end{equation}
holds for all non-negative real-valued measurable functions $f$, if and only if
\begin{equation}\label{D2}
0< D_{2}=\inf_{x\neq a}\mathcal{D}_{2}(|x|_{a})=\inf_{x\neq a}\left[\left(\int_{\X\setminus B(a,|x|_{a})}u(y)dy\right)^{\frac{1}{q}}\left(\int_{\X\setminus B(a,|x|_{a})}v^{1-p'}(y)dy\right)^{\frac{1}{p'}}\right],
\end{equation}
and $\mathcal{D}_{2}(|x|_{a})$ is  non-increasing.
Moreover, the largest constant $C_{2}(p,q)$ satisfies 
\begin{equation}
    D_{2}\geq C_{2}(p,q)\geq |p|^{\frac{1}{q}}(p')^{\frac{1}{p'}}D_{2},
\end{equation}
where $\frac{1}{p}+\frac{1}{p'}=1.$
\end{thm}
\begin{proof}
The main idea of the proof of this theorem is similar to that of  Theorem 2.2 with the only difference that $\mathcal{D}_{2}(|x|_{a})$ is  non-increasing, so we omit the details.
\end{proof}
\section{Consequences on homogeneous groups}
In this section, we consider several consequences of the main results for the reverse integral Hardy, Hardy-Littlewood-Sobolev and Stein-Weiss  inequalities on homogeneous groups.

 Let us recall that a Lie group (on $\mathbb{R}^{n}$) $\mathbb{G}$ with the dilation
$$D_{\lambda}(x):=(\lambda^{\nu_{1}}x_{1},\ldots,\lambda^{\nu_{n}}x_{n}),\; \nu_{1},\ldots, \nu_{n}>0,\; D_{\lambda}:\mathbb{R}^{n}\rightarrow\mathbb{R}^{n},$$
which is an automorphism of the group $\mathbb{G}$ for each $\lambda>0,$
is called a {\em homogeneous (Lie) group}. For simplicity, throughout this paper we use the notation $\lambda x$ for the dilation $D_{\lambda}(x).$  The homogeneous dimension of the homogeneous group $\mathbb{G}$ is denoted by $Q:=\nu_{1}+\ldots+\nu_{n}.$
Also, in this paper we denote a homogeneous quasi-norm on $\mathbb{G}$ by $|x|$, which
is a continuous non-negative function
\begin{equation}
\mathbb{G}\ni x\mapsto |x|\in[0,\infty),
\end{equation}
with the following properties

\begin{itemize}
	\item[i)] $|x|=|x^{-1}|$ for all $x\in\mathbb{G}$,
	\item[ii)] $|\lambda x|=\lambda |x|$ for all $x\in \mathbb{G}$ and $\lambda>0$,
	\item[iii)] $|x|=0$ if and only if $x=0$.
\end{itemize}
Let us also recall the following well-known fact about quasi-norms.
\begin{prop}[e.g. \cite{FR}, Proposition 3.1.38 and \cite{RS18}, Proposition 1.2.4] \label{prop_quasi_norm}
If $|\cdot|$ is a homogeneous quasi-norm on  $\mathbb{G}$, there exists $C>0$ such that for every $x,y\in \mathbb{G}$, we have
\begin{equation}\label{tri}
|x y|\leq C(|x| + |y|).
\end{equation}
\end{prop}
The following polarisation formula on homogeneous Lie groups will be used in our proofs:
there is a (unique)
positive Borel measure $\sigma$ on the
unit quasi-sphere
$
\S:=\{x\in \mathbb{G}:\,|x|=1\},
$
so that for every $f\in L^{1}(\mathbb{G})$ we have
\begin{equation}\label{EQ:polar}
\int_{\mathbb{G}}f(x)dx=\int_{0}^{\infty}
\int_{\S}f(ry)r^{Q-1}d\sigma dr.
\end{equation}
We refer to \cite{FS1} for the original appearance of such groups, to \cite{FR} and to \cite{RS18} for a recent comprehensive treatment.
Let us define quasi-ball centered at $x$ with radius $r$ in the following form:
\begin{equation}
B(x,r):=\{y\in\mathbb{G}:|x^{-1}y|<r\}.
\end{equation}

\subsection{Reverse integral Hardy inequality}
In this sub-section we show the reverse integral Hardy inequality on homogeneous Lie groups.
\begin{thm}
Let $\mathbb{G}$ be a homogeneous Lie group of homogeneous dimension $Q$ with a quasi-norm $|\cdot|$. Assume that $q\leq p<0$ and $\alpha,\beta\in\mathbb{R}$. Then the reverse
integral Hardy inequality
\begin{equation}\label{inhar1}
\left[\int_{\mathbb{G}}\left(\int_{B(0,|x|)}f(y)dy\right)^{q}|x|^{\alpha}dx\right]^{\frac{1}{q}}\geq C_{1}\left(\int_{\mathbb{G}}f^{p}(x)|x|^{\beta}dx\right)^{\frac{1}{p}},
\end{equation}
holds for some $C_{1}$ > 0 and for all non-negative measurable functions $f$, if  $\alpha+Q>0$, $\beta(1-p')+Q>0$ and $\frac{Q+\alpha}{q}+\frac{Q+\beta(1-p')}{p'}=0,$ where $\frac{1}{p}+\frac{1}{p'}=1$.
Moreover, the biggest constant $C_{1}$ for (3.4) satisfies
$$\left(\frac{|\S|}{\alpha+Q}\right)^{\frac{1}{q}}\left(\frac{|\S|}{Q+\beta(1-p')}\right)^{\frac{1}{p'}}\geq C_{1}\geq|p|^{\frac{1}{q}}(p')^{\frac{1}{p'}}\left(\frac{|\S|}{\alpha+Q}\right)^{\frac{1}{q}}\left(\frac{|\S|}{Q+\beta(1-p')}\right)^{\frac{1}{p'}}. $$
\end{thm}
\begin{proof}
Let us show that the condition (2.4) is satisfied with $u(x)=|x|^{\alpha}$ and $v(x)=|x|^{\beta}$. We calculate the first integral in (2.4):
\begin{equation}
\begin{split}
\int_{B(0,|x|)}u(y)dy&=\int_{B(0,|x|)}|y|^{\alpha}dy
\stackrel{(3.2)}=\int_{0}^{|x|}\int_{\S}r^{\alpha}r^{Q-1}dr d\sigma=\frac{|\S|}{Q+\alpha}|x|^{Q+\alpha},
\end{split}
\end{equation}
where $|\S|$ is the area of the unit quasi-sphere in $\mathbb{G}$.
Then,
\begin{equation*}
\begin{split}
\int_{B(0,|x|)}v^{1-p'}(y)dy&=\int_{B(0,|x|)}|y|^{\beta(1-p')}dy\\&
\stackrel{(3.2)}=\int_{0}^{|x|}\int_{\S}r^{\beta(1-p')}r^{Q-1}dr d\sigma\\&=\frac{|\S|}{Q+\beta(1-p')}|x|^{Q+\beta(1-p')}.
\end{split}
\end{equation*}
Finally, by using above facts and $\frac{Q+\alpha}{q}+\frac{Q+\beta(1-p')}{p'}=0$, we have
\begin{equation*}
\begin{split}
\mathcal{D}_{1}(|x|)&=\left(\frac{|\S|}{\alpha+Q}\right)^{\frac{1}{q}}\left(\frac{|\S|}{Q+\beta(1-p')}\right)^{\frac{1}{p'}}\left[|x|^{\frac{Q+\alpha}{q}+\frac{Q+\beta(1-p')}{p'}}\right]=\left(\frac{|\S|}{\alpha+Q}\right)^{\frac{1}{q}}\left(\frac{|\S|}{Q+\beta(1-p')}\right)^{\frac{1}{p'}},
\end{split}
\end{equation*}
which shows that $\mathcal{D}_{1}(|x|)$ is a non-decreasing function. Then 

$$D_{1}=\inf_{x\neq a}\mathcal{D}_{1}(|x|)=\left(\frac{|\S|}{\alpha+Q}\right)^{\frac{1}{q}}\left(\frac{|\S|}{Q+\beta(1-p')}\right)^{\frac{1}{p'}}>0.$$
Therefore, by (2.5) we have 
$$D_{1}\geq C_{1}\geq|p|^{\frac{1}{q}}(p')^{\frac{1}{p'}}D_{1},$$
where $D_{1}=\left(\frac{|\S|}{\alpha+Q}\right)^{\frac{1}{q}}\left(\frac{|\S|}{Q+\beta(1-p')}\right)^{\frac{1}{p'}}$ thereby,  completing the proof.
\end{proof}
Now we obtain the  conjugate reverse integral Hardy inequality on homogeneous Lie
groups.
\begin{thm}
Let $\mathbb{G}$ be a homogeneous Lie group of homogeneous dimension $Q$ with a quasi-norm $|\cdot|$. Assume that $q\leq p<0$ and $\alpha,\beta\in\mathbb{R}$. Then the reverse conjugate
integral Hardy inequality
\begin{equation}\label{inhar2}
\left[\int_{\mathbb{G}}\left(\int_{\mathbb{G}\setminus B(0,|x|)}f(y)dy\right)^{q}|x|^{\alpha}dx\right]^{\frac{1}{q}}\geq C_{2}\left(\int_{\mathbb{G}}f^{p}(x)|x|^{\beta}dx\right)^{\frac{1}{p}},
\end{equation}
holds for some $C_{2}>0$ and for all non-negative measurable functions $f$, if  $\alpha+Q<0$, $\beta(1-p')+Q<0$ and $\frac{Q+\alpha}{q}+\frac{Q+\beta(1-p')}{p'}=0.$
Moreover, the biggest constant $C_{2}$ for (3.6) satisfies
$$\left(\frac{|\S|}{|\alpha+Q|}\right)^{\frac{1}{q}}\left(\frac{|\S|}{|Q+\beta(1-p')|}\right)^{\frac{1}{p'}}\geq C_{2}\geq|p|^{\frac{1}{q}}(p')^{\frac{1}{p'}}\left(\frac{|\S|}{|\alpha+Q|}\right)^{\frac{1}{q}}\left(\frac{|\S|}{|Q+\beta(1-p')|}\right)^{\frac{1}{p'}}. $$
\end{thm}
\begin{proof}
Proof of this theorem is similar to the previous case, where we use Theorem 2.4 instead of of Theorem 2.2.
\end{proof}
\subsection{The reverse Hardy-Littlewood-Sobolev inequality and Stein-Weiss  inequality}

In this sub-section we obtain the reverse Hardy-Littlewood-Sobolev inequality and Stein-Weiss  inequality on Euclidean space and homogeneous Lie groups. 

Let us introduce the Riesz operator on homogeneous Lie groups in the following form:
\begin{equation}
I_{\lambda,|\cdot|}u(x)=|x|^{\lambda}*u=\int_{\mathbb{G}}|y^{-1} x|^{\lambda}u(y)dy,\,\,\,\lambda<0,
\end{equation}
where $*$ is the convolution. Hence, by taking $\G=(\mathbb{R}^{n},+)$, $Q=n$ and $|\cdot|=|\cdot|_{E}$ ($|\cdot|_{E}$ is the Euclidean distance), we get the Riesz operator on Euclidean space:
\begin{equation}
I_{\lambda,|\cdot|_{E}}u(x)=|x|_{E}^{\lambda}*u=\int_{\mathbb{G}}|x-y|_{E}^{\lambda}u(y)dy,\,\,\,\lambda<0.
\end{equation}
Firstly, let us present the Hardy-Littlewood-Sobolev inequality on Euclidean space.
\begin{thm}[The reverse Hardy-Littlewood-Sobolev inequality on $\mathbb{R}^{n}$]\label{hlstheoremeuc}
Assume that $n\geq1$, $q< p<0$, $\lambda<0$ such that $\frac{1}{p'}+\frac{1}{q}+\frac{\lambda}{n}=0$, where $\frac{1}{p}+\frac{1}{p'}=1$ and $\frac{1}{q}+\frac{1}{q'}=1$. Then for all non-negative functions $f\in L^{q'}(\mathbb{R}^{n})$ and $0<\int_{\mathbb{R}^{n}}h^{p}(x)dx<\infty$, we get
\begin{equation}\label{hlseuc}
    \int_{\mathbb{R}^{n}}\int_{\mathbb{R}^{n}}f(x)|x-y|_{E}^{\lambda}h(y)dxdy\geq C\left(\int_{\mathbb{R}^{n}}f^{q'}(x)dx\right)^{\frac{1}{q'}}\left(\int_{\mathbb{R}^{n}}h^{p}(x)dx\right)^{\frac{1}{p}},
\end{equation}
where $C$ is a positive constant independent of $f$ and $h$.
\end{thm}
\begin{proof}
By using the reverse H\"{o}lder inequality with $\frac{1}{q}+\frac{1}{q'}=1$, we calculate
\begin{equation*}
\begin{split}
 \int_{\mathbb{R}^{n}}\int_{\mathbb{R}^{n}}f(x)|x-y|_{E}^{\lambda}h(y)dydx
\stackrel{(2.2)}\geq\left(\int_{\R}\left(\int_{\R}|x-y|_{E}^{\lambda}h(y)dy\right)^{q}dx\right)^{\frac{1}{q}}\|f\|_{L^{q'}(\R)}.
\end{split}
\end{equation*}
Thus for (3.9), it is enough to show that
\begin{equation*}
\left(\int_{\mathbb{R}^{n}}\left(\int_{\R}|x-y|_{E}^{\lambda}h(y)dy\right)^{q}dx\right)^{\frac{1}{q}}\geq C\left(\int_{\R}h^{p}(x)dx\right)^{\frac{1}{p}}.
\end{equation*}
By direct calculation, we have
\begin{equation}
\left(\int_{\R}\left(\int_{\R}|x-y|_{E}^{\lambda}h(y)dy\right)^{q}dx\right)^{\frac{1}{q}}
\stackrel{q<0}\geq\left(\int_{\R}\left(\int_{B_{E}\left(0,|x|_{E}\right)}|x-y|_{E}^{\lambda}h(y)dy\right)^{q}dx\right)^{\frac{1}{q}},
\end{equation}
where $B_{E}(0,|x|_{E})$ is the Euclidean ball centered at $0$ with radius $|x|_{E}$.
By using $|y|_{E}\leq|x|_{E}$, we get
\begin{equation}
|x-y|_{E}\leq |x|_{E} +|y|_{E}\leq 2|x|_{E}.
\end{equation}
Then for any $\lambda<0$, we have
\begin{equation}
    \begin{split}
\left(\int_{\R}\left(\int_{\R}|x-y|_{E}^{\lambda}h(y)dy\right)^{q}dx\right)^{\frac{1}{q}}
&\stackrel{q<0}\geq\left(\int_{\R}\left(\int_{B_{E}\left(0,|x|_{E}\right)}|x-y|_{E}^{\lambda}h(y)dy\right)^{q}dx\right)^{\frac{1}{q}}\\&
\geq 2^{\lambda}\left(\int_{\R}|x|_{E}^{\lambda q}\left(\int_{B_{E}\left(0,|x|_{E}\right)}h(y)dy\right)^{q}dx\right)^{\frac{1}{q}}.
    \end{split}
\end{equation}
If condition (2.4) in Theorem 2.2 with $u(x)=|x|^{\lambda q}$ and $v(x)=1$ in (2.3) is satisfied, then we have
$$\left(\int_{\R}|x|_{E}^{\lambda q}\left(\int_{B_{E}\left(0,|x|_{E}\right)}h(y)dy\right)^{q}dx\right)^{\frac{1}{q}}\geq C\left(\int_{\R}h^{p}(x)dx\right)^{\frac{1}{p}}.$$
Let us show that the condition (2.4) is satisfied. From the assumption, we have
\begin{equation}
  0=\frac{1}{p'}+\frac{1}{q}+\frac{\lambda}{n}\stackrel{\frac{1}{p'}>0}> \frac{1}{q}+\frac{\lambda}{n},
\end{equation}
which means $n+\lambda q>0.$
By using this fact,  we obtain
\begin{equation}
    \begin{split}
        \int_{{B_{E}\left(0,|x|_{E}\right)}}u(y)dy&=\int_{{B\left(0,|x|_{E}\right)}}|y|_{E}^{\lambda q}dy\\&
\stackrel{(3.2)}=\int_{0}^{|x|_{E}}\int_{\S}r^{\lambda q}r^{n-1}dr d\sigma\\&
=\frac{|\S|}{n+\lambda q}|x|_{E}^{n+\lambda q},
    \end{split}
\end{equation}
and
\begin{equation}
    \begin{split}
       \int_{{B_{E}\left(0,|x|_{E}\right)}}v^{1-p'}(y)dy= \int_{{B_{E}\left(0,|x|_{E}\right)}}1dy=|\S||x|_{E}^{n}.
    \end{split}
\end{equation}
Finally, by using the assumption $\frac{1}{p'}+\frac{1}{q}+\frac{\lambda}{n}=0$,
\begin{equation}
    \begin{split}
        \mathcal{D}_{1}(|x|_{E})=\left(\frac{|\S|}{n+\lambda q}\right)^{\frac{1}{q}}\left(|\S|\right)^{\frac{1}{p'}}|x|_{E}^{\frac{n}{p'}+\frac{n+\lambda q}{q}}=\left(\frac{|\S|}{n+\lambda q}\right)^{\frac{1}{q}}|\S|^{\frac{1}{p'}},
    \end{split}
\end{equation}
which implies, $\mathcal{D}_{1}(|x|_{E})$ is a non-decreasing function. Thus, 
$$D_{1}=\inf_{x\neq 0}\mathcal{D}_{1}(|x|_{E})=\left(\frac{|\S|}{n+\lambda q}\right)^{\frac{1}{q}}|\S|^{\frac{1}{p'}}>0,$$
completing the proof.
\end{proof}
\begin{rem}
Inequality (3.10) seems  to be new even in the  Euclidean space.
\end{rem}
Also, let us now present the reverse Hardy-Littlewood-Sobolev inequality on $\G$. 
\begin{thm}[The reverse Hardy-Littlewood-Sobolev inequality on $\G$]\label{hlstheorem}
Let $\mathbb{G}$ be a homogeneous Lie group of homogeneous dimension $Q\geq1$ with arbitrary quasi-norm $|\cdot|$. Assume that $q< p<0$, $\lambda<0$ such that $\frac{1}{p'}+\frac{1}{q}+\frac{\lambda}{Q}=0$, where $\frac{1}{p}+\frac{1}{p'}=1$ and $\frac{1}{q}+\frac{1}{q'}=1$. Then for all non-negative functions $f\in L^{q'}(\mathbb{G})$ and $0<\int_{\G}h^{p}(x)dx<\infty$, we get
\begin{equation*}\label{hls}
    \int_{\G}\int_{\G}f(x)|y^{-1}x|^{\lambda}h(y)dxdy\geq C\left(\int_{\G}f^{q'}(x)dx\right)^{\frac{1}{q'}}\left(\int_{\G}h^{p}(x)dx\right)^{\frac{1}{p}},
\end{equation*}
where $C$ is a positive constant independent of $f$ and $h$.
\end{thm}
\begin{proof}
The proof of this theorem is similar to Theorem \ref{hlstheoremeuc}, but here we use Proposition \ref{prop_quasi_norm} and the polar decomposition formula \eqref{EQ:polar}.
\end{proof}

 Let us now show the reverse Stein-Weiss  inequality on $\R$.
\begin{thm}[The reverse Stein-Weiss inequality on $\R$]\label{stein-weiss5euc}

Assume that $n\geq1$, $q\leq p<0$, $\lambda<0$,  and $\frac{1}{p'}+\frac{1}{q}+\frac{\alpha+\beta+\lambda}{n}=0$, where $\frac{1}{p}+\frac{1}{p'}=1$ and $\frac{1}{q}+\frac{1}{q'}=1$. Then for all non-negative functions $f\in L^{q'}(\R)$ and $0<\int_{\R}h^{p}(x)dx<\infty$,
we have
\begin{equation}\label{stein-weiss124euc}
    \int_{\R}\int_{\R}|x|_{E}^{\alpha}f(x)|x-y|_{E}^{\lambda}h(y)|y|_{E}^{\beta}dxdy\geq C\left(\int_{\R}f^{q'}(x)dx\right)^{\frac{1}{q'}}\left(\int_{\R}h^{p}(x)dx\right)^{\frac{1}{p}},
\end{equation}
if  one of the following conditions is satisfied:
\begin{itemize}
	\item[(a)] $\beta>-\frac{n}{p'}$;

	\item[(b)]  $ \alpha>-\frac{n}{q}$.

\end{itemize}
\end{thm}

\begin{proof}
Similarly to Theorem 3.3, by using the reverse H\"{o}lder inequality with $\frac{1}{q}+\frac{1}{q'}=1$, we calculate
\begin{equation*}
\begin{split}
\int_{\R}\int_{\R}|x|_{E}^{\alpha}f(x)|x-y|_{E}^{\lambda}&h(y)|y|_{E}^{\beta}dydx=\int_{\R}\left(\int_{\R}|x|_{E}^{\alpha}|x-y|_{E}^{\lambda}h(y)|y|_{E}^{\beta}dy\right)f(x)dx\\&
\geq\left(\int_{\R}\left(\int_{\R}|x|_{E}^{\alpha}|x-y|_{E}^{\lambda}h(y)|y|_{E}^{\beta}dy\right)^{q}dx\right)^{\frac{1}{q}}\|f\|_{L^{q'}(\R)}.
\end{split}
\end{equation*}
Thus for (3.18), it is enough to show that
\begin{equation*}
\left(\int_{\R}\left(\int_{\R}|x|_{E}^{\alpha}|x-y|_{E}^{\lambda}h(y)|y|_{E}^{\beta}dy\right)^{q}dx\right)^{\frac{1}{q}}\geq C\left(\int_{\R}h^{p}(x)dx\right)^{\frac{1}{p}},
\end{equation*}
and by substituting $z(y)=h(y)|y|_{E}^{\beta}$, this is equivalent to
\begin{equation*}
\int_{\R}\left(\int_{\R}|x|_{E}^{\alpha}|x-y|_{E}^{\lambda}z(y)dy\right)^{q}dx\leq C\left(\int_{\R}|y|_{E}^{-\beta p}z^{p}(x)dx\right)^{\frac{q}{p}}.
\end{equation*}
We have that
\begin{equation*}
\int_{\R}|x|_{E}^{\alpha}|x-y|_{E}^{\lambda}z(y)dy\geq\int_{B_{E}\left(0,|x|_{E}\right)}|x|_{E}^{\alpha}|x-y|_{E}^{\lambda}z(y)dy,
\end{equation*}
then
\begin{equation*}
\left(\int_{\R}|x|_{E}^{\alpha}|x-y|_{E}^{\lambda}z(y)dy\right)^{q}\stackrel{q<0}\leq\left(\int_{B_{E}\left(0,|x|_{E}\right)}|x|_{E}^{\alpha}|x-y|_{E}^{\lambda}z(y)dy\right)^{q}.
\end{equation*}
Therefore, we obtain
\begin{multline}\label{i1}
\left(\int_{\R}|x|_{E}^{\alpha q}\left(\int_{\R}|x-y|_{E}^{\lambda}z(y)dy\right)^{q}dx\right)^{\frac{1}{q}}\\
\stackrel{q<0}\geq\left(\int_{\R}|x|_{E}^{\alpha q}\left(\int_{B_{E}\left(0,|x|_{E}\right)}|x-y|_{E}^{\lambda}z(y)dy\right)^{q}dx\right)^{\frac{1}{q}}:=I^{\frac{1}{q}}_{1}.
\end{multline}
Similarly to (3.19), we have

\begin{multline}\label{i3}
\left(\int_{\R}|x|_{E}^{\alpha q}\left(\int_{\R}|x-y|_{E}^{\lambda}z(y)dy\right)^{q}dx\right)^{\frac{1}{q}}\\
\stackrel{q<0}\geq\left(\int_{\R}|x|_{E}^{\alpha q}\left(\int_{\R\setminus B_{E}(0,|x|_{E})}|x-y|_{E}^{\lambda}z(y)dy\right)^{q}dx\right)^{\frac{1}{q}}:=I^{\frac{1}{q}}_{2}.
\end{multline}
From (3.19)-(3.20), we obtain
\begin{equation}\label{i13}
\left(\int_{\R}|x|_{E}^{\alpha q}\left(\int_{\R}|x-y|_{E}^{\lambda}z(y)dy\right)^{q}dx\right)^{\frac{1}{q}}\geq I^{\frac{1}{q}}_{1},
\end{equation}
and 
\begin{equation}\label{i14}
\left(\int_{\R}|x|_{E}^{\alpha q}\left(\int_{\R}|x-y|_{E}^{\lambda}z(y)dy\right)^{q}dx\right)^{\frac{1}{q}}\geq I^{\frac{1}{q}}_{2}.
\end{equation}

\textbf{Step 1.} Let us prove (a) for (3.21).  By using $|y|_{E}\leq |x|_{E}$, we get
\begin{equation}
|x-y|_{E}\leq |x|_{E} +|y|_{E}\leq 2|x|_{E}.
\end{equation}
Then for any $\lambda<0$, we have
$$2^{\lambda}|x|_{E}^{\lambda}\leq |x-y|_{E}^{\lambda}.$$
Therefore, we get
\begin{equation*}
I_{1}=\int_{\R}|x|_{E}^{\alpha q}\left(\int_{B_{E}\left(0,|x|_{E}\right)}|x-y|_{E}^{\lambda}z(y)dy\right)^{q}dx
\leq 2^{\lambda q}\int_{\R}|x|_{E}^{(\alpha+\lambda)q}\left(\int_{B_{E}\left(0,|x|_{E}\right)}z(y)dy\right)^{q}dx.
\end{equation*}
If condition (2.4) in Theorem 2.2 with $u(x)=|x|_{E}^{(\alpha+\lambda)q}$ and $v(y)=|y|_{E}^{-\beta p}$ in (2.3) is satisfied, then we have
\begin{equation*}
I_{1}\leq C\int_{\R}\left(\int_{B_{E}\left(0,|x|_{E}\right)}z(y)dy\right)^{q}|x|_{E}^{(\alpha+\lambda)q}dx
\leq C\left(\int_{\R}|y|_{E}^{-\beta p}z^{p}(y)dy\right)^{\frac{q}{p}}.
\end{equation*}
Let us verify that the condition (2.4) holds. By using the assumption $\beta>-\frac{n}{p'}$, we obtain
$$0=\frac{1}{p'}+\frac{1}{q}+\frac{\alpha+\beta+\lambda}{n}>\frac{1}{q}+\frac{\alpha+\lambda}{n},$$
that is, $\frac{n+(\alpha+\lambda)q}{nq}<0,$ or $n+(\alpha+\lambda)q>0,$. Then, we get
\begin{equation*}
\begin{split}
\left(\int_{B_{E}\left(0,|x|_{E}\right)}u(y)dy\right)^{\frac{1}{q}}&=
\left(\int_{B_{E}\left(0,|x|_{E}\right)}|y|_{E}^{(\alpha+\lambda)q}dy\right)^{\frac{1}{q}}
=\left(\frac{|\mathfrak{S}|}{n+(\alpha+\lambda)q)}\right)^{\frac{1}{q}} |x|^{\frac{n+(\alpha+\lambda)q}{q}}.
\end{split}
\end{equation*}
Since $\beta>-\frac{n}{p'}$, we have
$$-\beta p(1-p')+n=\beta p'+n>0.$$
Thus $-\beta p(1-p')+n>0$. Then, a direct computation gives
\begin{equation}
\begin{split}
\left(\int_{B_{E}\left(0,|x|_{E}\right)}v^{1-p'}(y)dy\right)^{\frac{1}{p'}}&=\left(\int_{ B_{E}\left(0,|x|_{E}\right)}|y|_{E}^{-\beta p(1-p')}dy\right)^{\frac{1}{p'}}\\&= \left(\frac{|\mathfrak{S}|}{\beta p'+n}\right)^{\frac{1}{p'}} |x|_{E}^{\frac{\beta p'+n}{p'}}.
\end{split}
\end{equation}
Therefore by using $\frac{1}{p'}+\frac{1}{q}+\frac{\alpha+\beta+\lambda}{n}=0$, we have
\begin{equation*}\label{a1}
\begin{split}
\mathcal{D}_{1}(|x|_{E})&=\left(\int_{B_{E}\left(0,|x|_{E}\right)}u(y)dy\right)^{\frac{1}{q}}\left(\int_{ B_{E}\left(0,|x|_{E}\right)}v^{1-p'}(y)dy\right)^{\frac{1}{p'}}\\&
=\left(\frac{|\mathfrak{S}|}{n+(\alpha+\lambda)q}\right)^{\frac{1}{q}}\left(\frac{|\mathfrak{S}|}{\beta p'+n}\right)^{\frac{1}{p'}},
\end{split}
\end{equation*}
which means $\mathcal D_{1}(|x|_{E})$ is a non-decreasing function. Therefore,
\begin{equation}
D_{1}=\inf_{x\neq 0} \mathcal D_{1}(|x|_{E})=\left(\frac{|\mathfrak{S}|}{n+(\alpha+\lambda)q}\right)^{\frac{1}{q}}\left(\frac{|\mathfrak{S}|}{\beta p'+n}\right)^{\frac{1}{p'}} >0.
\end{equation}
Then by using (2.3), we obtain
\begin{equation}\label{I1}
I^{\frac{1}{q}}_{1}\geq  C\left(\int_{\R}|y|_{E}^{-\beta p}z^{p}(y)dy\right)^{\frac{1}{p}}=C\left(\int_{\R}h^{p}(y)dy\right)^{\frac{1}{p}}.
\end{equation}

\textbf{Step 2.} Let us prove (b) for (3.22). From $|x|_{E}\leq |y|_{E}$,  we calculate
$$|x-y|_{E}\leq |x|_{E}+|y|_{E}\leq 2|y|_{E},$$
then
$$ |x-y|_{E}^{\lambda}\geq C|y|_{E}^{\lambda},$$
where $C>0$.
Then, if condition (2.22) with $u(x)=|x|_{E}^{\alpha q}$ and $v(y)=|y|_{E}^{-(\beta+\lambda)p}$ is satisfied, then we have
\begin{equation*}
I_{2}
\leq C\int_{\R}|x|_{E}^{\alpha q}\left(\int_{\R\setminus B_{E}(0,|x|_{E})}z(y)|y|_{E}^{\lambda}dy\right)^{q}dx
\leq C\left(\int_{\R}|y|_{E}^{-\beta p}z^{p}(y)dy\right)^{\frac{q}{p}}.
\end{equation*}
Now let us check that the condition (2.22) holds. We have
\begin{equation*}
\begin{split}
\left(\int_{\R\setminus B_{E}(0,|x|_{E})}u(y)dy\right)^{\frac{1}{q}}=\left(\int_{\R\setminus B_{E}(0,|x|_{E})}|y|_{E}^{\alpha q}dy\right)^{\frac{1}{q}}
&=\left(\int^{\infty}_{|x|_{E}}\int_{\mathfrak{S}}r^{\alpha q}r^{n-1}drd\sigma\right)^{\frac{1}{q}}\\&
= \left(\frac{|\mathfrak{S}|}{|n+\alpha q|}\right)^{\frac{1}{q}} |x|_{E}^{\frac{n+\alpha q}{q}},
\end{split}
\end{equation*}
where $n+\alpha q<0$. From $\alpha>-\frac{n}{q}$, we have $0=\frac{1}{p'}+\frac{1}{q}+\frac{\alpha+\beta+\lambda}{n}>\frac{1}{p'}+\frac{\beta+\lambda}{n},$
then
\begin{equation}\label{betalambda}
(\beta+\lambda)p'+n<0.
\end{equation}
By using this fact, we have
\begin{equation*}
\begin{split}
\left(\int_{\R\setminus B_{E}(0,|x|_{E})}v^{1-p'}(y)dy\right)^{\frac{1}{p'}}&=\left(\int_{\R\setminus B_{E}(0,|x|)}|y|_{E}^{-(\beta+\lambda)(1-p')p}dy\right)^{\frac{1}{p'}}\\&
= \left(\frac{|\mathfrak{S}|}{|n+(\beta+\lambda)p'|}\right)^{\frac{1}{p'}} |x|_{E}^{\frac{n+(\beta+\lambda)p'}{p'}}.
\end{split}
\end{equation*}
Then by using $\frac{1}{p'}+\frac{1}{q}+\frac{\alpha+\beta+\lambda}{n}=0$, we get
\begin{equation}
\begin{split}
    \mathcal{D}_{2}(|x|_{E})&=\left(\frac{|\mathfrak{S}|}{|n+\alpha q|}\right)^{\frac{1}{q}}\left(\frac{|\mathfrak{S}|}{|n+(\beta+\lambda)p'|}\right)^{\frac{1}{p'}},
\end{split}
\end{equation}
which means $\mathcal{D}_{2}(|x|_{E})$ is a non-increasing function.
Therefore, we have
\begin{equation}\label{a2}
\begin{split}
D_{2}&=\inf_{x\neq 0}\mathcal{D}_{2}(|x|_{E})=
\left(\frac{|\mathfrak{S}|}{|n+\alpha q|}\right)^{\frac{1}{q}}\left(\frac{|\mathfrak{S}|}{|n+(\beta+\lambda)p'|}\right)^{\frac{1}{p'}}>0.
\end{split}
\end{equation}
Then, we have
\begin{equation}\label{I3}
I^{\frac{1}{q}}_{2}\geq C\left(\int_{\R}|y|_{E}^{-\beta p}z^{p}(y)dy\right)^{\frac{1}{p}}=C\left(\int_{\R}h^{p}(y)dy\right)^{\frac{1}{p}}.
\end{equation}
\end{proof}
\begin{rem}
Inequality (3.18) seems  to be new even in the  Euclidean space.
\end{rem}
Let us now show the reverse Stein-Weiss  inequality $\G$.
\begin{thm}[The reverse Stein-Weiss inequality on $\G$]\label{stein-weiss5}

Let $\mathbb{G}$ be a homogeneous group of homogeneous dimension $Q\geq1$ and let $|\cdot|$ be an arbitrary homogeneous quasi-norm on $\mathbb{G}$.
Assume that $q\leq p<0$, $\lambda<0$,  and $\frac{1}{p'}+\frac{1}{q}+\frac{\alpha+\beta+\lambda}{Q}=0$, where $\frac{1}{p}+\frac{1}{p'}=1$ and $\frac{1}{q}+\frac{1}{q'}=1$. Then for all non-negative functions $f\in L^{q'}(\mathbb{G})$ and $0<\int_{\G}h^{p}(x)dx<\infty$,
we have
\begin{equation}\label{stein-weiss124}
    \int_{\G}\int_{\G}|x|^{\alpha}f(x)|y^{-1}x|^{\lambda}h(y)|y|^{\beta}dxdy\geq C\left(\int_{\G}f^{q'}(x)dx\right)^{\frac{1}{q'}}\left(\int_{\G}h^{p}(x)dx\right)^{\frac{1}{p}},
\end{equation}
if  one of the following conditions is satisfied:
\begin{itemize}
	\item[(a)] $\beta>-\frac{Q}{p'}$;

	\item[(b)]  $ \alpha>-\frac{Q}{q}$.

\end{itemize}
\end{thm}

\begin{proof}
The proof of similar to the previous theorem,  but here we use Proposition \ref{prop_quasi_norm} and the polar decomposition formula \eqref{EQ:polar}.
\end{proof}

%%%%%%%%%%%%%%% End of first page %%%%%%%%%%%%%%%%%%%%%

%%%%%%%%%% Insert bibliography here %%%%%%%%%%%%%%

\end{document}